\documentclass[12pt,amscd]{amsart}
\usepackage[all]{xy}
\usepackage{graphicx,tikz-cd}
\usepackage{amsmath,amsxtra,amssymb,latexsym, amscd,amsthm}
\usepackage{indentfirst}
\usepackage[mathscr]{eucal}
\usepackage[pagebackref=true]{hyperref}

\newtheorem{thm}{Theorem}[section]

\newtheorem{lem}[thm]{Lemma}

\theoremstyle{definition}
\newtheorem{defn}[thm]{Definition}
\newtheorem{exm}[thm]{Example}

\newtheorem{rem}[thm]{Remark}

\numberwithin{equation}{section}

\DeclareMathOperator{\Ind}{Ind}

\begin{document}

\title{Contractible independence complexes of trees}

\author[My Hanh Pham]{My Hanh Pham}
\address{Faculty of Education, An Giang University\\Vietnam National University Ho Chi Minh City\\ Long Xuyen, An Giang, Vietnam}
\email{pmhanh@agu.edu.vn}

\author{Thanh Vu}
\address{Institute of Mathematics, VAST, 18 Hoang Quoc Viet, Hanoi, Vietnam}
\email{vuqthanh@gmail.com}

\subjclass[2020]{05C31, 05C69}
\keywords{independence complex; independence polynomial; tree}

\date{}

\commby{}

\begin{abstract}
   We show that the independence complex of a tree is contractible if and only if it can be reduced to a path \( P_n \) with \( n \equiv 1 \pmod{3} \) by a sequence of truncation moves at branching points. As a consequence of our method, we also characterize the trees for which the independence polynomial evaluated at \( -1 \) is equal to \( 1 \) or \( -1 \).
\end{abstract}

\maketitle

\section{Introduction}
\label{sect_intro}

Let \( G \) be a finite simple graph. An \textit{independent set} in \( G \) is a set of pairwise nonadjacent vertices. The maximum cardinality of an independent set in \( G \) is called the \textit{independence number} of \( G \), denoted by \( \alpha(G) \). The \textit{independence polynomial} of \( G \) is defined by
\[
I(G;x) = s_0 + s_1 x + \dots + s_{\alpha(G)} x^{\alpha(G)},
\]
where \( s_i \) is the number of independent sets of size \( i \) in \( G \).

The independence polynomial of a graph was first studied by Gutman and Harary \cite{GH}, and has since been investigated by many others; see \cite{AMSE, G, HLMP, HL} and the references therein. For a survey, we refer to \cite{LM1}. There are only a few classes of graphs for which the independence polynomial is computed explicitly \cite{HKV, LM3}. For the independence polynomial of trees, one central open problem is unimodality. This property is known to hold for the class of claw-free graphs \cite{CS}.

The \textit{independence complex} of $G$ is the simplicial complex whose faces are the independent sets of $G$. Consequently, the independence polynomial of a graph is the $f$-polynomial of its independence complex. By results of Kozlov \cite{Ko} and Ehrenborg and Hetyei \cite{EH}, the independence complex of a tree is either contractible or homotopy equivalent to a sphere. In particular, \( I(G; -1) \in \{-1, 0, 1\} \) for any tree \( G \), and the independence complex of \( G \) is contractible if and only if \( I(G; -1) = 0 \). Kim \cite{Ki} recently proved that the independence complex of a graph with no induced cycle of length divisible by $3$ is either contractible or homotopy equivalent to a sphere. Ehrenborg and Hetyei also determined the homotopy type of the independence complex of path graphs. In particular, \( \Ind(P_n) \), the independence complex of a path on \( n \) vertices, is contractible if and only if \( n \equiv 1 \pmod{3} \). Engstr\"om \cite{E} characterized trees for which the independence complex is contractible. This was later generalized by Adamaszek \cite{A} and Faridi and Holleben \cite{FH}. In this paper, we completely characterize trees for which \( I(G;-1) = 0, 1, \) or \( -1 \). In particular, we show that \( \Ind(G) \) is contractible if and only if it can be reduced to a path \( P_n \) with \( n \equiv 1 \pmod{3} \) by a sequence of truncation moves at branching points.

We now introduce some terminology to make the previous statement more precise. Let $u$ be a vertex of $G$. The \textit{neighborhood} of $u$ in $G$, denoted by $N_G(u)$, is the set of all vertices $v$ such that $\{u,v\}$ is an edge in $G$. The \textit{degree} of $u$ is the cardinality of $N_G(u)$. A vertex $u$ of $G$ is called a \textit{branching point} if $\deg_G(u) > 2$, and it is called a \textit{leaf} if $\deg_G(u) = 1$. 

Let $G$ be a simple graph and $u$ be a branching point of $G$. A path starting at $u$ and ending at a leaf is called a \textit{pure branch} if every vertex on the path, except for $u$, has degree at most $2$ in $G$. We define the type of a branch $P$ to be the residue of its length modulo $3$. We define the following moves at branching points:

\begin{enumerate}
    \item If there exist a branch of type $1$ and a branch of type $2$, remove all vertices of $G$ except those belonging to these two branches.
    \item If there is more than one pure branch of the same type, remove all but one such branch.
    \item If $P$ is a pure branch of $G$ at $u$ of type $0$, remove all vertices on this branch except $u$.    
\end{enumerate}

\begin{defn}
Let $G$ be a simple graph. Let $H$ be a graph obtained by one of the moves above. Then $H$ is said to be a reduction of $G$, and we say that $G$ is reducible to $H$.
\end{defn}

The characterization of trees with a contractible independence complex is as follows.

\begin{thm}\label{thm_contractible} 
Let $G$ be a tree. Then $\Ind(G)$ is contractible if and only if $G$ is reducible to a path $P_n$ with $n \equiv 1 \pmod{3}$.
\end{thm}

To characterize graphs with \( I(G;-1) = 1 \) or \( -1 \), we need to take parity into account; hence, we introduce an additional definition.

\begin{defn}
Let \( G \) be a simple graph and let \( u \) be a branching point of \( G \). Assume that \( u \) has \( t \ge 2 \) pure branches of the same type, and let \( M \) be the move obtained by removing \( t-1 \) of these branches. We say that the move \( M \) is \emph{odd} if there is an odd number of branches whose length has residue \( 4 \pmod{6} \) when the type is \( 1 \), and residue \( 2 \pmod{6} \) when the type is \( 2 \). Otherwise, \( M \) is called an \emph{even} move.

Let \( M \) be the move of removing \( t \) pure branches of type \( 0 \) at \( u \). Then \( M \) is called an \emph{odd} move if an odd number of these branches have length congruent to \( 3 \pmod{6} \); otherwise, it is called an \emph{even} move.
\end{defn}

\begin{thm}\label{thm_pm1}
Let \( G \) be a tree. Then the following statements hold:
\begin{enumerate}
    \item \( I(G; -1) = 1 \) if and only if \( G \) is reducible to \( P_n \) with \( n \equiv 5,0 \pmod{6} \) by an even number of odd moves, or to \( P_n \) with \( n \equiv 2,3 \pmod{6} \) by an odd number of odd moves;
    \item \( I(G; -1) = -1 \) if and only if \( G \) is reducible to \( P_n \) with \( n \equiv 5,0 \pmod{6} \) by an odd number of odd moves, or to \( P_n \) with \( n \equiv 2,3 \pmod{6} \) by an even number of odd moves.
\end{enumerate}
\end{thm}

\section{Independence polynomial of trees at $-1$}\label{sec_pre}
Let \( G = (V(G), E(G)) \) be a simple graph with vertex set \( V(G) \) and edge set \( E(G) \). Let \( A \subseteq V(G) \). We denote by \( G - A \) the induced subgraph of \( G \) on \( V(G) \setminus A \). When \( A = \{u\} \), we write \( G - u \) instead of \( G - \{u\} \). The closed neighborhood of \( u \) in \( G \) is \( N_G[u] = N_G(u) \cup \{u\} \).

We have the following standard lemma \cite{GH}.

\begin{lem}\label{techn}
The following assertions hold:
\begin{enumerate}
    \item If \( G_1 \) and \( G_2 \) are disjoint graphs, then \( I(G_1 \cup G_2; x) = I(G_1; x)\, I(G_2; x) \).
    \item For every \( v \in V(G) \), we have \( I(G; x) = I(G - v; x) + x\, I(G - N_G[v]; x) \).
\end{enumerate}
\end{lem}

The independence polynomial of paths is well known \cite{LM2}. In particular, 

\begin{lem}\label{paths}
Let \( P_n \) be a path on \( n \) vertices. Then
\[
I(P_n; -1) =
\begin{cases}
    1 & \text{if } n \equiv 0, 5 \pmod{6},\\
    0 & \text{if } n \equiv 1, 4 \pmod{6},\\
    -1 & \text{if } n \equiv 2, 3 \pmod{6}.
\end{cases}
\]
\end{lem}

\begin{lem}\label{lem_mix_1_2}
Let \( G \) be a simple graph and let \( u \) be a branching point of \( G \). Assume that \( u \) has a pure branch of type \( 1 \) and a pure branch of type \( 2 \). Then \( I(G;-1)=0 \).
\end{lem}

\begin{proof}
Let \( P = u, v_1, \ldots, v_{3k+1} \) and \( Q = u, w_1, \ldots, w_{3\ell+2} \) be pure branches at \( u \) of type \( 1 \) and type \( 2 \), respectively, where \( k, \ell \) are nonnegative integers. By assumption, \( G - u \) has a connected component consisting of \( v_1, \ldots, v_{3k+1} \), which is isomorphic to \( P_{3k+1} \). Moreover, \( G - N_G[u] \) has a connected component \( \{w_2, \ldots, w_{3\ell+2}\} \), which is isomorphic to \( P_{3\ell+1} \). By Lemma \ref{techn} and Lemma \ref{paths}, \( I(G-u;-1) = 0 \) and \( I(G-N_G[u];-1) = 0 \). Applying Lemma~\ref{techn}, we conclude that \( I(G;-1) = 0 \).
\end{proof}

\begin{lem}\label{lem_pure_type_1}
Let \( G \) be a simple graph and let \( u \) be a branching point of \( G \). Assume that \( u \) has \( t \ge 2 \) pure branches of type \( 1 \). Let \( H \) be the graph obtained by removing all but one such branch. Assume that among the removed branches, there are \( s \) branches whose length has residue \( 4 \pmod{6} \). Then
\[
I(G;-1) = (-1)^s I(H;-1).
\]
\end{lem}
\begin{proof}
For \( j = 1, \ldots, t \), let the \( j \)-th pure branch from \( u \) of type \( 1 \) be
\( u, v_{j,1}, \ldots, v_{j,3k_j+1} \). Let
\[
H = G \setminus \{ v_{j,\ell} \mid j = 2, \ldots, t,\ \ell = 1, \ldots, 3k_j+1 \}
\]
be the graph obtained by removing \( t-1 \) pure branches of type \( 1 \) from \( G \).

Both \( G - u \) and \( H - u \) contain a connected component isomorphic to the path \( P_{3k_1+1} \); hence \( I(G-u;-1) = I(H-u;-1) = 0 \).

Now, \( G - N_G[u] \) is the disjoint union of \( H - N_H[u] \) and \( t-1 \) paths of lengths \( 3k_j - 1 \) for \( j \ge 2 \). By Lemma~\ref{techn} and Lemma~\ref{paths}, we obtain
\[
I(G-N_G[u];-1) = (-1)^s I(H-N_H[u];-1),
\]
where \( s \) is the number of branches among the removed ones whose length is congruent to \( 4 \pmod{6} \). Applying Lemma~\ref{techn}, we deduce that
\[
I(G;-1)
= - I(G-N_G[u];-1)
= - (-1)^s I(H-N_H[u];-1) 
= (-1)^s I(H;-1).
\]
The conclusion follows.
\end{proof}

\begin{lem}\label{lem_pure_type_2}
Let \( G \) be a simple graph and let \( u \) be a branching point of \( G \). Assume that \( G \) has \( t \ge 2 \) pure branches of type \( 2 \). Let \( H \) be the graph obtained by removing all but one such branch. Assume that among the removed branches, there are \( s \) branches whose length has residue \( 2 \pmod{6} \). Then
\[
I(G;-1) = (-1)^s I(H;-1).
\]
\end{lem}

\begin{proof}
For \( j = 1, \ldots, t \), let the \( j \)-th pure branch from \( u \) of type \( 2 \) be
\( u, v_{j,1}, \ldots, v_{j,3k_j+2} \).
Let
\[
H = G \setminus \{ v_{j,\ell} \mid j = 2, \ldots, t,\ \ell = 1, \ldots, 3k_j+2 \}
\]
be the graph obtained by removing \( t-1 \) pure branches of type \( 2 \) from \( G \).

Both \( G - N_G[u] \) and \( H - N_H[u] \) contain a connected component isomorphic to the path \( P_{3k_1+1} \); hence
\[
I(G-N_G[u];-1) = I(H-N_H[u];-1) = 0.
\]

Now, \( G - u \) is the disjoint union of \( H - u \) and \( t-1 \) paths of lengths \( 3k_j+1 \) for \( j \ge 2 \). By Lemma~\ref{techn} and Lemma~\ref{paths}, we obtain
\[
I(G-u;-1) = I(H-u;-1)\,(-1)^s,
\]
where \( s \) is the number of removed branches whose length is congruent to \( 2 \pmod{6} \). Applying Lemma~\ref{techn}, we deduce that
\[
I(G;-1)
= I(G-u;-1)
= (-1)^s I(H-u;-1)
= (-1)^s I(H;-1).
\]
The conclusion follows.
\end{proof}

\begin{lem}\label{lem_pure_type_3}
Let \( G \) be a simple graph and let \( u \) be a branching point of \( G \). Let \( H \) be the graph obtained by removing all pure branches of type \( 0 \). Assume that among these branches, there are \( s \) whose length has residue \( 3 \pmod{6} \). Then
\[
I(G;-1) = (-1)^s I(H;-1).
\]
\end{lem}

\begin{proof}
Let \( P \) be a pure branch from \( u \) of length \( 6k + r \), where \( k \ge 0 \) and \( r \in \{3,6\} \). It suffices to prove that
\[
I(G;-1) = (-1)^{r/3} I(H;-1),
\]
where \( H \) is obtained from \( G \) by removing the branch \( P \). Indeed, we have
\[
G - u = (H - u) \cup (P - u)
\quad \text{and} \quad
G - N_G[u] = (H - N_H[u]) \cup (P - N_P[u]).
\]
Note that \( P - u \cong P_{6k+r} \) and \( P - N_P[u] \cong P_{6k+r-1} \). By Lemma~\ref{paths}, we obtain
\[
I(P_{6k+r};-1) = I(P_{6k+r-1};-1) = (-1)^{r/3}.
\]
Applying Lemma~\ref{techn}, we deduce the desired conclusion.
\end{proof}

We now show that any tree can be reduced to \( P_n \) for some \( n \).

\begin{lem}\label{lem_reducible_trees}
Let \( G \) be a tree. Then \( G \) is reducible to \( P_n \) for some \( n \ge 1 \).
\end{lem}
\begin{proof}
We proceed by induction on \( n \), the number of vertices of \( G \). The base case \( n \le 2 \) is trivial.

Note that if \( G \) has no branching points, then it is isomorphic to \( P_n \). Hence, it suffices to show that if \( G \) has a branching point, then it contains a branching point with at least two pure branches. Let \( v \) be a leaf of \( G \). Consider a path from \( v \) to another leaf of \( G \), and let \( w \) be the branching point closest to \( v \), so that the branch from \( w \) to \( v \) is pure. Let \( H \) be the graph obtained by removing this branch. If \( H \) is a path, then \( w \) is a branching point with at least two branches. Otherwise, by induction, \( H \) has a branching point with at least two pure branches.

If this branching point is \( w \) itself, then we are done. Otherwise, let this branching point of \( H \) be \( w_1 \). If the pure branches of \( H \) at \( w_1 \) do not contain \( w \), then they are also pure branches of \( G \). Hence, we are done if \( w_1 \) has two pure branches not containing \( w \). If \( w_1 \) has a pure branch containing \( w \), note that \( w \) cannot be a leaf of \( G \), so this pure branch must extend beyond \( w \). This extended path from \( w \) is a pure branch of \( G \) at \( w \). Hence, \( w \) has at least two pure branches. The conclusion follows.
\end{proof}

\begin{proof}[Proof of Theorem \ref{thm_contractible}]
The conclusion follows from Lemmas \ref{lem_reducible_trees}, \ref{lem_mix_1_2}, \ref{lem_pure_type_1}, \ref{lem_pure_type_2}, and \ref{lem_pure_type_3}.
\end{proof}

\begin{proof}[Proof of Theorem \ref{thm_pm1}]
By Lemma \ref{lem_mix_1_2}, we have \( I(G;-1) = \pm 1 \) if and only if no mixed branches of types \(1\) and \(2\) appear simultaneously at any branching point during the reduction process. The conclusion then follows from Lemmas \ref{lem_reducible_trees}, \ref{lem_pure_type_1}, \ref{lem_pure_type_2}, and \ref{lem_pure_type_3}.
\end{proof}

\begin{exm} Consider the following graph and its branch removal transformations. 

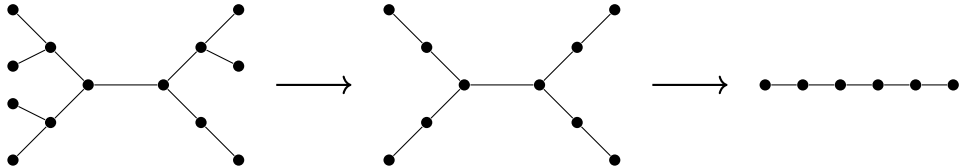
\begin{center}
    
\begin{tikzpicture}[
    scale=0.5,
    every node/.style={circle, fill, inner sep=1.5pt}
]

\begin{scope}[shift={(0,0)}]

\node (1) at (0,0) {};
\node (8) at (2,0) {};
\draw (1) -- (8);

\node (2) at (-1,1) {};
\node (3) at (-2,2) {};
\node (4) at (-2,0.5) {};
\draw (1) -- (2);
\draw (2) -- (3);
\draw (2) -- (4);

\node (5) at (-1,-1) {};
\node (6) at (-2,-0.5) {};
\node (7) at (-2,-2) {};
\draw (1) -- (5);
\draw (5) -- (6);
\draw (5) -- (7);

\node (9) at (3,1) {};
\node (10) at (4,2) {};
\node (11) at (4,0.5) {};
\draw (8) -- (9);
\draw (9) -- (10);
\draw (9) -- (11);

\node (12) at (3,-1) {};
\node (13) at (4,-2) {};
\draw (8) -- (12);
\draw (12) -- (13);

\end{scope}

\draw[->, thick] (5,0) -- (7,0);

\begin{scope}[shift={(10,0)}]

\node (1b) at (0,0) {};
\node (8b) at (2,0) {};
\draw (1b) -- (8b);

\node (2b) at (-1,1) {};
\node (3b) at (-2,2) {};
\draw (1b) -- (2b);
\draw (2b) -- (3b);

\node (5b) at (-1,-1) {};
\node (7b) at (-2,-2) {};
\draw (1b) -- (5b);
\draw (5b) -- (7b);

\node (9b) at (3,1) {};
\node (10b) at (4,2) {};
\draw (8b) -- (9b);
\draw (9b) -- (10b);

\node (12b) at (3,-1) {};
\node (13b) at (4,-2) {};
\draw (8b) -- (12b);
\draw (12b) -- (13b);

\end{scope}

\draw[->, thick] (15,0) -- (17,0);

\begin{scope}[shift={(21,0)}]

\node (A) at (-3,0) {};
\node (B) at (-2,0) {};
\node (C) at (-1,0) {};
\node (D) at (0,0) {};
\node (E) at (1,0) {};
\node (F) at (2,0) {};

\draw (A) -- (B) -- (C) -- (D) -- (E) -- (F);

\end{scope}

\end{tikzpicture}

\end{center}
The original graph is transformed into the second graph by collapsing two branches of length \(1\) at three branching points; these are even moves. The final graph is obtained by collapsing two branches of length \(2\) at two branching points; these are odd moves. Hence, \( G \) is reducible to \( P_6 \) by an even number of odd moves. Therefore, we deduce that \( I(G; -1) = 1 \).
\end{exm}

\begin{rem}
The order in which pure branches are removed is not important; the moves may be performed in any order. For each removal of pure branches of type \(1\) or \(2\), we may retain the branch of smallest length. Note that, in the previous three lemmas on branch removal, the graph on which the removal is performed need not be a tree, and the results still hold. This will be useful for understanding the value of the independence polynomial of a graph at $-1$ beyond trees.
\end{rem}

\subsection*{Acknowledgments}
We are grateful to Professor Alexander Engstr\"om and Thiago Holleben for pointing out relevant work on the homotopy of the independence complexes of graphs.


\begin{thebibliography}{2}


\bibitem {A} M. Adamaszek,  \emph{A note on independence complexes of chordal graphs and dismantling}, The electronic journal of combinatorics {\bf 24(2)} (2017), P2.34.


\bibitem {AMSE} Y. Alavi, P. J. Malde, A. J. Schwenk, P. Erd\"{o}s,  \emph{The vertex independence sequence of a graph is not constrained}, Congressus Numerantium, \textbf{58} (1987), 15--23.


\bibitem {CS} M. Chudnovsky, P. Seymour,  \emph{The roots of the independence polynomial of a claw-free graph},  Journal of Combinatorial Theory, Series B, \textbf{97} (2007), 350--357.




\bibitem {EH} R. Ehrenborg, G. Hetyei,  \emph{The topology of the independence complex}, European Journal of Combinatorics, \textbf{27} (2006), 906--923.



\bibitem {E} A. Engstr\"om,  \emph{Complexes of directed trees and independence complexes}, Discrete Mathematics \textbf{309} (2009), 3299--3309.


\bibitem {FH} S. Faridi and T. Holleben, \emph{Spherical complexes}, arXiv:2311.07727

\bibitem {G} I. Gutman, \emph{An identity for the independence polynomials of trees}, Publications de L'Intitut Mathematique, Nouvelle serie  {\bf 50(64)} (1991), 19--23.


\bibitem {GH} I. Gutman, F. Harary, \emph{Generalizations of the matching polynomial}, Utilitas Mathematica \textbf{24} (1983) 97--106.


\bibitem{HLMP} D. T. Hoang, V. E. Levit, E. Mandrescu, M. H. Pham, \emph{Log-concavity of the indepence polynomial of $W_p$ graphs}, Discrete Mathematics {\bf 349} (2026), 115109.

\bibitem{HL} C. Hoede and X. Li, \textit{Clique polynomials and independent set polynomials of graphs}. Discrete Mathematics {\bf 125} (1994), 219--228.




\bibitem{HKV} Hibi, T., Kara, S., and Vien, D. (2026). Independence polynomials of graphs. arXiv:2603.16695.


\bibitem {Ki} J. Kim, \emph{The homotopy type of the independence complex of graphs with no induced cycles of length divisible by 3}, European Journal of Combinatorics {\bf 104} (2022), 103534.


\bibitem {Ko} D. Kozlov, \emph{Complexes of Directed Trees}, J. Com. Theory, Series A, \textbf{88} (1999), 112--122.


\bibitem{LM1}  V. E. Levit, E. Mandrescu, \emph{The independence polynomial of a graph-a survey}.   Proceedings of the 1st International Conference on Algebraic Informatics, Aristotle Univ. Thessaloniki Thessaloniki. (2005), 233--254. 




 


\bibitem{LM2} V. E. Levit and E. Mandrescu. The independence polynomial of a graph at $-1$. arXiv:0904.4819.

 

 
\bibitem{LM3}V. E. Levit, E. Mandrescu, \emph{On the  independence polynomial of the corona of graphs}, Discrete Applied Mathematics {\bf 203} (2016), 85--93.





  


\end{thebibliography}
\end{document}